\documentclass[12pt]{article}
\usepackage[utf8]{inputenc}
\usepackage{amssymb, amsmath,amsthm, thmtools, mathabx, mathtools}
\usepackage{bbm}
\usepackage{amsfonts}
\usepackage[left=3cm,right=3cm,top=2.5cm,bottom=2.5cm]{geometry}
\usepackage{xcolor}

\usepackage{tabto}
\TabPositions{2 cm, 4 cm, 6 cm, 8 cm}

\usepackage{tikz-cd}

\usepackage{tikz}
\usetikzlibrary{arrows, automata, positioning}

\usepackage{tocloft}
\usepackage{appendix}

\usepackage{enumerate}

\usepackage{hyperref}
\usepackage[capitalize]{cleveref}

\hypersetup{
  colorlinks   = true, 
  urlcolor     = blue, 
  linkcolor    = blue, 
  citecolor   = blue 
}

\newtheorem{theorem}{Theorem}[section]
\newtheorem{corollary}[theorem]{Corollary}
\newtheorem{lemma}[theorem]{Lemma}

\newtheorem{proposition}[theorem]{Proposition}

\newtheorem*{theorem*}{Theorem}

\theoremstyle{definition}
\newtheorem{definition}[theorem]{Definition}
\newtheorem{example}[theorem]{Example}
\newtheorem{remark}[theorem]{Remark}

\newtheorem*{remark*}{Remark}
\newtheorem*{notation*}{Notation}
\newtheorem*{acks*}{Acknowledgements}
\newtheorem*{out*}{Outline}
\newtheorem*{disc*}{AI disclosure}

\theoremstyle{plain}
\newtheorem*{mainthm}{Main Theorem}

\renewcommand\leq{\leqslant}
\renewcommand\geq{\geqslant}

\newcommand{\N}{\mathbb{N}}
\newcommand{\Z}{\mathbb{Z}}

\newcommand{\R}{\mathbb{R}}

\newcommand{\cl}{\operatorname{cl}}
\newcommand{\scl}{\operatorname{scl}}

\newcommand{\normal}[1]{\left<\! \left< #1\right> \!\right>}

\begin{document}

\title{Profinite isomorphisms, stable commutator length, and fixed point properties}
\author{Francesco Fournier-Facio}
\date{\today}
\maketitle

\begin{abstract}
We construct Grothendieck pairs witnessing that  stable commutator length and quasimorphisms are not profinite invariants, answering a question of Echtler and Kammeyer. Moreover, we recover the fact that property FA, property NL (which obstructs actions on hyperbolic spaces), and property FW$_\infty$ (which obstructs actions on finite-dimensional CAT(0) cube complexes) and the existence of non-abelian free subgroups are not profinite invariants. The method combines Rips constructions with iterated group-theoretic Dehn filling on hyperbolic virtually special groups.
\end{abstract}

\section{Introduction}

The topic of \emph{profinite rigidity}, which has been a central focus of group theory for several years \cite{reid:survey, bridson:survey, profinite:survey}, asks to what extent properties of a finitely generated residually finite group $G$ can be detected by its set of finite quotients, equivalently by its profinite completion $\widehat{G}$. A group property $\mathcal{P}$ is a \emph{profinite invariant} if whenever $G$ and $H$ are finitely generated residually finite groups with $\widehat{G} \cong \widehat{H}$, then $G$ has $\mathcal{P}$ if and only if $H$ has $\mathcal{P}$.
There are a few positive results (see Section 5 of the surveys \cite{bridson:survey, profinite:survey}), even for certain isomorphism types (see e.g. \cite{pickel2, kleinian, coxeter, julian1}). However most results are negative: property (T) \cite{profinite:T}, fixed point properties \cite{profinite:FA, bridson:fix}, finiteness properties \cite{profinite:FP}, conjugacy separability \cite{profinite:conjsep}, $\ell^2$-invariants \cite{profinite:l2}, orderability \cite{profinite:everything, profinite:bi}, amenability \cite{profinite:amenable} and second bounded cohomology \cite{profinite:hb} are all known to not be profinite invariants.

\begin{mainthm}
\label{intro:main}
There exists a finitely generated residually finite group $G$, with a finitely generated normal subgroup $N$, such that $N \to G$ induces an isomorphism of profinite completions, and the following hold.
\begin{itemize}
\item $N$ has vanishing $\scl$, no unbounded quasimorphisms, no non-abelian free subgroups, it cannot act on a hyperbolic space with a loxodromic element, and it cannot act on a finite-dimensional CAT(0) cube complex without a global fixed point.
\item $G$ has unbounded $\scl$, an infinite-dimensional space of homogeneous quasimorphisms, and admits a general type action on a tree.
\end{itemize}
\end{mainthm}

It is common for counterexamples to profinite invariance to arise via inclusions of subgroups, these are called \emph{Grothendieck pairs} after \cite{grothendieck}. We now discuss the results in more detail.

\subsection*{Stable commutator length and quasimorphisms}

For a group $G$ and an element $g \in G$, its \emph{commutator length} is
\[\cl_G(g) \coloneqq \inf \{ \ell \mid g = [a_1, b_1] \cdots [a_\ell, b_\ell] : a_i, b_i \in G \} \in \N \cup \{ \infty \};\]
and its \emph{stable commutator length} is
\[\scl_G(g) \coloneqq \liminf\limits_{n \to \infty} \frac{\cl_G(g^n)}{n} \in \R_{\geq 0} \cup \{ \infty \}.\]
This is the subject of a rich theory, with algebraic, geometric and analytic aspects \cite{calegari}. Bavard duality \cite{bavard} connects $\scl$ to \emph{quasimorphisms}, that is functions $f \colon G \to \mathbb{R}$ such that
\[\sup\limits_{g, h} |f(g) + f(h) - f(gh)| < \infty.\]
A quasimorphism is \emph{homogeneous} if it restricts to a homomorphism on every cyclic subgroup. A special case of Bavard duality states that every homogeneous quasimorphism on $G$ is a homomorphism if and only if $\scl_G(g) \in \{0, \infty \}$ for all $g \in G$. Quasimorphisms feature prominently in other subjects, such as bounded cohomology \cite{frigerio}, knot theory \cite{knot}, symplectic geometry \cite{symp} and one-dimensional dynamics \cite{ghys}.

\medskip

Our main theorem shows that both the vanishing of $\scl$ and the existence of unbounded quasimorphisms are not profinite invariants. The latter answers a question of Echtler--Kammeyer \cite[page 381]{profinite:hb}. By contrast, the space of real-valued homomorphisms is a profinite invariant \cite[Proposition 5.7]{profinite:survey}.

\cref{norms} below ensures the vanishing of more general stable lengths on $N$, such as stable $w$-length \cite{stablewlength}: except in trivial cases, these are not profinite invariants (\cref{stablewlength}). Quasimorphisms are closely related to second bounded cohomology, which is known not to be a profinite invariant \cite{profinite:hb}, however this is independent of our result (\cref{h2b}).

\subsection*{Fixed point properties}

Following \cite{NL}, we say that a group has \emph{property NL (No Loxodromics)} if it cannot act on a hyperbolic space with a loxodromic element. Following \cite{guido, cornulier, NL:V}, we say that a group has \emph{property FW$_\infty$} if every action on a finite-dimensional CAT(0) cube complex has a global fixed point. We refer the reader to \cref{s:fixedpoints} for more on these properties.

\medskip

Our main theorem shows that neither property NL nor property FW$_\infty$ are profinite invariants, both of which strengthen the fact that property FA is not a profinite invariant \cite{profinite:FA}. In fact, although this is not stated, the construction of \cite{profinite:FA}, already shows that NL and FW$_\infty$ are not profinite invariants: the authors construct groups $\Gamma$ and $\Delta$, with the same profinite completion, such that $\Gamma$ does not have FA (hence does not have NL or FW$_\infty$) while $\Delta$ has (T), hence in particular FA and FW$_\infty$ (and even FW, obstructing actions on all CAT(0) cube complexes \cite{niblo:reeves}). Moreover $\Delta$ is a lattice in a product of two high rank simple algebraic groups over local fields, so it has NL by a result of Haettel \cite{NL:lattices1} or a more recent result of Bader, Caprace, Furman and Sisto \cite{NL:lattices2}. Our construction, unlike theirs, is a Grothendieck pair. However, Grothendieck pairs witnessing the profinite non-invariance of property FA (as well as fixed point properties for actions on CAT(0) spaces of fixed dimension) have been constructed in \cite{bridson:fix}.

We also recover the fact that the existence of non-abelian free subgroups is not a profinite invariant. This had already been obtained in \cite{profinite:amenable}, however our examples are of a completely different nature: most notably, the subgroup $N$ may be chosen to be non-amenable (\cref{freesub}).

\subsection*{Method}

Our initial motivation was to prove that quasimorphisms are not a profinite invariant. Proofs of failure of profinite invariance usually follow one of the following routes, none of which can be directly used to tackle quasimorphisms.

\begin{itemize}
\item The first approach, commonly associated with the work of Platonov--Tavgen \cite{platonov:tavgen} and Bridson--Grunewald \cite{bridson:grunewald} uses Rips-like constructions, often\footnote{But not always, see \cite{profinite:cohopf} for a recent example.} combined with fibre products. The groups obtained this way map onto (acylindrically) hyperbolic groups, hence have an infinite-dimensional space of homogeneous quasimorphisms \cite{epstein:fujiwara, bestvina:fujiwara}.

\item The second approach uses lattices in high rank Lie groups, see e.g. \cite{profinite:T}. In this case, every quasimorphism is bounded \cite{burger:monod1, burger:monod2, iozzi:appendix}.

\item The third approach uses branch groups, see e.g. \cite{profinite:amenable}. Such examples are often either amenable or torsion, and both conditions imply the vanishing of $\scl$. Certain dynamical approaches to second bounded cohomology vanishing \cite{kotschick, cc} suggest that this might hold in general, although it remains an open question.
\end{itemize}

Moreover, there are many constructions of families of profinitely isomorphic but non-isomorphic groups belonging to certain prescribed classes: solvable groups \cite{baumslag, pickel, solvable}, Neumann-type diagonal products \cite{pyber}, torus bundles \cite{funar} (all in the amenable world, thus with vanishing $\scl$), non-amenable Seifert-fibred manifolds \cite{hempel} (which have an infinite-dimensional space of homogeneous quasimorphisms, factoring through the orbifold quotient) and again, branch groups \cite{nekr} (see above).

\medskip

Therefore our construction needs a new input, and this will be \emph{iterated group-theoretic Dehn filling} on hyperbolic virtually special groups. Similar ingredients have already been used, in a different way, in proofs of failure of profinite invariance: \cite{bridson:special} exploits that the output of the Rips construction is hyperbolic and virtually special, and \cite{kleinian:fp, BR:notes} use an iterated small cancellation construction to produce large families of groups to feed into a fibre product construction.

We will now sketch a proof of the main theorem; for concreteness we focus on quasimorphisms and $\scl$, which as we mentioned above are two faces of the same coin \cite{bavard}. It will be clear to the expert that the method is very flexible and allows to impose several properties on the groups (see also \cref{flexibility}). We thus hope that it will be useful for more applications.

\medskip

We begin similarly to the first approach above, by taking as input a suitable version of the Rips construction \cite{rips, wise:rips, macarena}. This gives a short exact sequence
\[1 \to N_1 \to G_1 \to Q \to 1,\]
where $G_1$ is hyperbolic and virtually special, hence residually finite \cite{special}. Choosing an appropriate $Q$, we may ensure that $N_1 \to G_1$ induces an isomorphism of profinite completions, and that $Q$ has an infinite-dimensional space of homogeneous quasimorphisms. Here $G_1$ is hyperbolic, so $N_1$ is acylindrically hyperbolic, and therefore both have an infinite-dimensional space of homogeneous quasimorphisms \cite{bestvina:fujiwara}.

To mend this, we apply group-theoretic Dehn filling to produce quotient short exact sequences
\[1 \to N_i \to G_i \to Q \to 1,\]
using the Malnormal Special Quotient Theorem \cite{wise:book, MSQT} to ensure that each $G_i$ is hyperbolic and virtually special. The added relations are chosen so that $\scl$ decreases on larger and larger balls in the normal subgroups. In the limit, we obtain a short exact sequence
\[1 \to N_\infty \to G_\infty \to Q \to 1\]
such that $N_\infty$ has vanishing $\scl$, hence no unbounded quasimorphisms \cite{bavard}, and $G_\infty$ is residually finite. The properties of $Q$ still ensure that $N_\infty \to G_\infty$ induces an isomorphism of profinite completions, and that $G_\infty$ has an infinite-dimensional space of homogeneous quasimorphisms, factoring through $Q$.

\begin{out*}
In \cref{s:prelim} we recall some background, basic facts and constructions. In \cref{s:norms} we present a first version of the construction, focusing on quasimorphisms and $\scl$, as in the sketch above. In \cref{s:NL} we refine the construction to prove the main theorem. In \cref{s:remarks} we discuss additional features and limitations of our method, as well as the relation of our results to the literature.
\end{out*}

\begin{acks*}
This work was undertaken in Cambridge, where the author was supported by the Herchel-Smith Postdoctoral Fellowship Fund. He thanks Alan Reid for encouraging him to work on the profinite invariance of quasimorphisms, as well as Macarena Arenas, Harry Baik, Daniel Echtler, Sam Fisher, Holger Kammeyer, Martin Kassabov, Henry Wilton and Julian Wykowski for useful conversations.
\end{acks*}

\begin{disc*}
An LLM was used to proofread the final manuscript for typos. It was not used in the research or in the preparation of the mathematical content.
\end{disc*}

\section{Preliminaries}
\label{s:prelim}

\subsection{Group-theoretic Dehn filling}

The inductive step in our construction will use group-theoretic Dehn filling in relatively hyperbolic groups \cite{osin, groves:manning}, most importantly the Malnormal Special Quotient Theorem \cite{wise:book}, or rather its strengthened version \cite{MSQT}. We assume some familiarity with hyperbolicity. We will only use one example of relatively hyperbolic pair.

\begin{theorem}[{\cite[Theorem 7.11]{bowditch}}]
\label{bowditch}
Let $G$ be a torsion-free hyperbolic group and let $P < G$ be an infinite, malnormal, quasiconvex subgroup. Then $G$ is hyperbolic relative to $P$.
\end{theorem}

To find subgroups $P$ satisfying the hypotheses, we use the following well-known fact, see e.g. \cite[Proposition 7.3]{tarski}.

\begin{proposition}
\label{random}
Let $G$ be a torsion-free hyperbolic group, and let $H < G$ be a non-elementary subgroup. For every primitive element $g \in G$ and every $n \in \N$ there exist $h_1, \ldots, h_n \in H$ such that $\{ g, h_1, \ldots, h_n \}$ is the basis of a malnormal quasiconvex free subgroup of $G$.
\end{proposition}

If $G$ is hyperbolic relative to $P$, and $W$ is a normal subgroup of $P$, the corresponding \emph{group-theoretic Dehn filling} is $G/\normal{W}$, where $\normal{W}$ denotes the normal closure of $W$ in $G$.

Finally, we recall that a group is \emph{cocompactly cubulated} if it is the fundamental group of a compact non-positively curved cube complex\footnote{Or equivalently if the group acts \emph{freely} and cocompactly on a CAT(0) cube complex, in particular we do not allow for finite stabilisers, for convenience.}, and it is \emph{special} if this cube complex is special in the sense of Haglund--Wise \cite{special}\footnote{Some authors use \emph{compact special} to stress the compactness hypothesis, but we follow the convention of \cite[Definition 2.1]{MSQT}}.

\begin{remark}
\label{ccc:summary}

These properties will only be used as tools. Let us just recall:
\begin{itemize}
\item Non-positively curved cube complexes are aspherical, hence cocompactly cubulated groups are finite-dimensional, and in particular torsion-free;
\item By Agol's Theorem \cite{agol}, a hyperbolic cocompactly cubulated group is virtually special;
\item Virtually special groups are residually finite \cite{special}.
\end{itemize}
\end{remark}

The fundamental example for our purposes is that of classical small cancellation groups \cite[Chapter V]{CGT}.

\begin{theorem}[\cite{gromov, wise}]
\label{sc}
A finitely presented $C'(\frac{1}{6})$-small cancellation group, with no relation a proper power, is hyperbolic and cocompactly cubulated.
\end{theorem}

Here is the version of the Malnormal Special Quotient Theorem that we will use.

\begin{theorem}[{\cite[Theorem 2.7]{MSQT}}]
\label{MSQT}

Let $G$ be non-elementary hyperbolic and virtually special, and hyperbolic relative to $P$. Then there exists a finite-index normal subgroup $P' < P$ such that if $W < P$ is a normal subgroup contained in $P'$ and $P/W$ is non-elementary hyperbolic and virtually special, then $G/\normal{W}$ is hyperbolic and virtually special.
\end{theorem}

\subsection{Grothendieck pairs}

We will produce Grothendieck pairs via the following criterion.

\begin{theorem}[{\cite[Theorem A]{bridson}}]
\label{bridson}

Let $1 \to N \to G \to Q \to 1$ be a short exact sequence. If $Q$ is finitely presented, has no non-trivial finite quotients, and $H_2(Q; \Z) = 0$, then the inclusion $N \to G$ induces an isomorphism of profinite completions.
\end{theorem}

We will apply this together with the Rips construction, which given a finitely presented group $Q$, produces a small cancellation group $G$ that maps onto $Q$ with finitely generated kernel $N$ \cite{rips}. \cref{sc} says that $G$ is hyperbolic and cocompactly cubulated, hence residually finite (\cref{ccc:summary}). One can ensure residual finiteness of $G$ without appealing to Agol's Theorem \cite{wise:rips}. We will also need to ensure certain additional properties on $N$. To achieve everything at once, we will use Arenas' very flexible version of the Rips construction.

\begin{theorem}[\cite{macarena}]
\label{macarena}

Let $Q$ be a finitely presented group, and let $\Gamma$ be a cocompactly cubulated group. Then there exists a short exact sequence $1 \to N \to G \to Q \to 1$ such that $G$ is hyperbolic, cocompactly cubulated, and $N$ is a quotient of $\Gamma$.
\end{theorem}

Finally, we need a source of groups $Q$ to which \cref{bridson} applies, that moreover have a large space of quasimorphisms. One possibility is Higman's group \cite{higman}.

\begin{proposition}
\label{higman}
Higman's group
\[Q \coloneqq \langle a, b, c, d \mid aba^{-1} = b^2, bcb^{-1} = c^2, cdc^{-1} = d^2, dad^{-1} = a^2 \rangle\]
has no non-trivial finite quotients and $H_2(Q; \Z) = 0$. Moreover, $Q$ has a general type action on a tree, and an infinite-dimensional space of homogeneous quasimorphisms.
\end{proposition}

\begin{proof}
The fact that $Q$ has no non-trivial finite quotients is proved in Higman's original paper \cite{higman}. There it is also remarked that $Q$ splits as an amalgamated product, which leads to a general type action on a tree \cite{trees}, and to quasimorphisms \cite{fujiwara:tree}\footnote{One could also use that $Q$ is acylindrically hyperbolic \cite[Corollary 4.26]{minasyan:osin} and apply \cite{bestvina:fujiwara}.}. Its homology can be computed by a Mayer--Vietoris argument, which shows that $Q$ is, in fact, acyclic \cite[pages 11-12]{BDH}.
\end{proof}

\subsection{Norms}

As in previous iterated small cancellation constructions \cite{muranov, tarski}, to kill $\scl$ we only use some basic properties of it, so it does not hurt to consider a more general setting.

\begin{definition}
	A function $\nu \colon G \to \R_{\geq 0} \cup \{\infty\}$ is a \emph{norm} if for all $g, h \in G$:
	\begin{itemize}
	\item $\nu(g) = 0$ if and only if $g = 1$;
	\item $\nu(g) = \nu(g^{-1})$;
	\item $\nu(gh) \leq \nu(g) + \nu(h)$.
	\end{itemize}
	Its \emph{stabilisation} $s \nu$ is defined by
	\[s\nu(g) \coloneqq \liminf\limits_{n \to \infty} \frac{\nu(g^n)}{n} \in \R_{\geq 0} \cup \{ \infty \}.\]
\end{definition}

Note that $s\nu(g^n) = |n| s\nu(g)$ for all $g \in G$ and all $n \in \Z$.

\begin{example}
	Every group admits a \emph{trivial} norm
	\[
	\nu_{triv}(g) \coloneqq 
	\begin{cases}
	0 \text{ if } g = 1; \\
	\infty \text{ otherwise}.
	\end{cases}
	\]
\end{example}

To relate norms in different groups, we introduce the following notion.

\begin{definition}
	An \emph{intrinsic norm} is an assignment $G \to \nu_G$ of a norm to every group, such that for every morphism $\varphi \colon G \to H$, and every $g \in G$, it holds $\nu_G(g) \geq \nu_H(\varphi(g))$.
\end{definition}

Intuitively, intrinsic norms are defined purely in terms of group theory, the first example being commutator length\footnote{The zeroth example being the trivial norm.}. More generally, let $w$ be a reduced word in a free group $F_n$. This defines a word map $w \colon G^n \to G$ on any group. We can then define \emph{$w$-length} \cite{words} as the symmetric word length in $G$ with respect to the alphabet $\{ w(g_1, \ldots, g_n) : g_i \in G \}$. This is an intrinsic norm, and its stabilisation is called \emph{stable $w$-length} \cite{stablewlength}. Note that $w$-length is trivial as a norm if and only if $w$ is the empty word.

\begin{definition}
	An intrinsic norm $\nu$ is \emph{non-elementary} if it is non-trivial on some (equivalently, every) non-abelian free group $F$.
\end{definition}

Tautologically, for every non-empty reduced word $w$, the $w$-length is non-elementary.

\begin{lemma}
\label{w:finite}

Let $\nu$ be a non-elementary intrinsic norm. Then there exists a finitely presented $C'(\frac{1}{6})$-small cancellation group $\Gamma$, with no relation a proper power, such that $\nu_{\Gamma}(g) < \infty$ for all $g \in \Gamma$.
\end{lemma}

\begin{proof}
Start with $F_2 = \langle a, b \rangle$. By assumption, there exists an element $g \in F_2$ such that $0 < \nu(g) < \infty$. We add a relation of the form
\[a^{-1} \prod\limits_{i = 1}^{n} \gamma_i g \gamma_i^{-1};\]
and another one of the same form for $b$. Choosing the elements $\gamma_i$ appropriately, we can ensure that these relations satisfy the $C'(\frac{1}{6})$-small cancellation condition and that they are not proper powers. In the quotient $\Gamma$, the images of $a$ and $b$ are products of conjugates of the image of $g$, hence they have finite $\nu_\Gamma$. Since they generate $\Gamma$, we conclude.
\end{proof}

\section{A first construction}
\label{s:norms}

The goal of this section is to present the method in a more essential form, only focusing on quasimorphisms and $\scl$ (\cref{scl}). We start with a lemma which will be used as the inductive step.

\begin{lemma}
\label{indstep}
	Let $G$ be a torsion-free, non-elementary hyperbolic, virtually special group. Let $N < G$ be a non-trivial normal subgroup, let $K < G$ be a finite-index normal subgroup, and let $A \subset N$ be a finite subset. Let $\nu$ be an intrinsic norm such that $\nu_N(a) < \infty$ for all $a \in A$, and let $\delta > 0$.
	
	Then there exists a normal subgroup $M < G$ such that
	\begin{itemize}
	\item $M \subset N \cap K$;
	\item $G/M$ is torsion-free, non-elementary hyperbolic and virtually special;
	\item For every $a \in A$, it holds $s\nu_{N/M}(a) < \delta$.
	\end{itemize}
\end{lemma}

\begin{proof}
By induction, it suffices to prove this in the case when $A = \{a\}$ consists of a single element. Write $a = b^r$, where $b$ is primitive. Because $N$ is non-elementary, \cref{random} gives elements $x, y \in N$ such that $P \coloneqq \langle b, x, y \rangle < G$ is a free group of rank $3$ that is quasiconvex and malnormal in $G$. By \cref{bowditch}, $G$ is hyperbolic relative to $P$. Let $P' < P$ be the finite-index normal subgroup given by \cref{MSQT}.

Let $m \in \N$ be sufficiently large so that $a^m \in K \cap P'$. Choose parameters
\[p > 6; \quad q > \frac{p \cdot \nu_N(a)}{\delta}.\]
If $\gamma_1, \ldots, \gamma_p \in \langle x, y\rangle$ are elements defined by small cancellation words in $\{ x, y \}$ that are sufficiently long compared to $m$ and $q$, the word
\[w \coloneqq b^{-rmq} \prod\limits_{i = 1}^p \gamma_i b^{rm} \gamma_i^{-1} = a^{-mq} \prod\limits_{i = 1}^p \gamma_i a^m \gamma_i^{-1}\]
is a $C'(\frac{1}{6})$-small cancellation word in $P$, moreover it is not a proper power. Let $W \coloneqq \normal{w}_P$ and $M \coloneqq \normal{w}_G = \normal{W}_G$.

Because $K \cap P'$ is normal in $P$ and $a^m \in K \cap P'$, also $w \in K \cap P' < N$ and so $M \subset N \cap K$. By \cref{sc}, the group $P/W$ is non-elementary hyperbolic and virtually special, so \cref{MSQT} applies and shows that $G/M$ is hyperbolic and virtually special.
 
To see that $G/M$ is torsion-free, we could express this process in terms of cubical small cancellation theory and apply \cite{macarena:asphericity}. For ease of reference, we instead use geometric small cancellation theory; this is slightly outside the scope of the preliminaries but it is very standard so we only give a minimal argument citing the literature. By choosing $p$ to be sufficiently large in terms of the embedding $P < G$, we may ensure that the group $\langle w \rangle$ satisfies a suitable ``tight'' \cite[Definition 2.15]{tarski} geometric small cancellation condition in $G$, see \cite[Proposition 7.1 and Remark 7.2]{tarski}. Since $G$ is torsion-free and $w$ is ``tight'', the quotient $G/M$ is torsion-free \cite[Proposition 5.18]{ppq}.

It remains to show that $s\nu_{N/M}(a) < \delta$. Indeed, the added relation implies that, in $N/M$, the element $a^{mq}$ is a product of $p$ conjugates of $a^m$, hence
\begin{align*}
&s\nu_{N/M}(a^{mq}) \leq \nu_{N/M}(a^{mq}) \leq p \cdot \nu_{N/M}(a^m) \leq pm \cdot \nu_N(a)\\
\Rightarrow \quad &s\nu_{N/M}(a) \leq \frac{p}{q} \nu_N(a) < \delta,
\end{align*}
where we used subadditivity, homogeneity, conjugacy-invariance and monotonicity under quotients.
\end{proof}

\begin{theorem}
\label{norms}

Let $\nu$ be a non-elementary intrinsic norm (such as commutator length, or more generally $w$-length for a non-empty reduced word $w$), and let $s \nu$ be its stabilisation. Then there exists a torsion-free finitely generated residually finite group $G_\infty$, with a finitely generated normal subgroup $N_\infty$, such that:
\begin{itemize}
\item $N_\infty \to G_\infty$ induces an isomorphism of profinite completions;
\item $s\nu_{N_\infty}(g) = 0$ for all $g \in N_\infty$;
\item $G_\infty$ has an infinite-dimensional space of homogeneous quasimorphisms.
\end{itemize}
\end{theorem}

\begin{proof}
Let $\Gamma$ be given by \cref{w:finite}; it is hyperbolic and cocompactly cubulated (\cref{sc}). Let $Q$ be Higman's group (\cref{higman}). By \cref{macarena}, there exists a short exact sequence
\[1 \to N_1 \to G_1 \to Q \to 1,\]
where $G_1$ is hyperbolic and cocompactly cubulated, hence torsion-free, virtually special, and residually finite (\cref{ccc:summary}); and $N_1$ is a quotient of $\Gamma$. It follows from the choice of $\Gamma$ that on $N_1$, and every quotient thereof, every non-trivial element has finite $\nu$. We fix a finite generating set $S$ for $G_1$, and henceforth the length of an element in a quotient of $G_1$ will always be intended in the Cayley graph with generating set the image of $S$.

We construct by induction short exact sequences
\[\begin{tikzcd}
	& \vdots & \vdots & \vdots \\
	1 & {N_i} & {G_i} & Q & 1 \\
	1 & {N_{i+1}} & {G_{i+1}} & Q & 1 \\
	& \vdots & \vdots & \vdots \\
	1 & {N_\infty} & {G_\infty} & Q & 1
	\arrow[two heads, from=1-2, to=2-2]
	\arrow[two heads, from=1-3, to=2-3]
	\arrow[no head, from=1-4, to=2-4]
	\arrow[shift left, no head, from=1-4, to=2-4]
	\arrow[from=2-1, to=2-2]
	\arrow[from=2-2, to=2-3]
	\arrow[two heads, from=2-2, to=3-2]
	\arrow[from=2-3, to=2-4]
	\arrow["{\pi_i}", two heads, from=2-3, to=3-3]
	\arrow[from=2-4, to=2-5]
	\arrow[shift left, no head, from=2-4, to=3-4]
	\arrow[no head, from=2-4, to=3-4]
	\arrow[from=3-1, to=3-2]
	\arrow[from=3-2, to=3-3]
	\arrow[two heads, from=3-2, to=4-2]
	\arrow[from=3-3, to=3-4]
	\arrow[two heads, from=3-3, to=4-3]
	\arrow[from=3-4, to=3-5]
	\arrow[shift left, no head, from=3-4, to=4-4]
	\arrow[no head, from=3-4, to=4-4]
	\arrow[two heads, from=4-2, to=5-2]
	\arrow[two heads, from=4-3, to=5-3]
	\arrow[shift left, no head, from=4-4, to=5-4]
	\arrow[no head, from=4-4, to=5-4]
	\arrow[from=5-1, to=5-2]
	\arrow[from=5-2, to=5-3]
	\arrow[from=5-3, to=5-4]
	\arrow[from=5-4, to=5-5]
\end{tikzcd}\]
and finite-index normal subgroups $K_i < G_i$ with the following properties:
\begin{enumerate}
\item\label{list:hvs} Each $G_i$ is torsion-free, hyperbolic and virtually special (hence residually finite);
\item\label{list:scl} For every element $a \in N_i$ of length at most $i$, it holds $s\nu_{N_i}(a) < \frac{1}{i}$;
\item\label{list:K} $\ker(\pi_{i-1}) < N_{i-1} \cap K_{i-1}$ and $K_i < \pi_{i-1}(K_{i-1})$;
\item\label{list:rf} The ball of radius $i$ in $G_i$ maps injectively into $G_i/K_i$.
\end{enumerate}
Let $i \geq 1$, and suppose by induction that $G_j, N_j, K_j$ have been constructed for $j < i$, and $G_i, N_i$ have been constructed, such that points \eqref{list:hvs}-\eqref{list:rf} are satisfied. By residual finiteness, there exists a finite-index normal subgroup $K_i < G_i$ such that the ball of radius $i$ maps injectively into $G_i/K_i$, giving \eqref{list:rf}; moreover, up to taking a deeper subgroup, we may assume that $K_i < \pi_{i-1}(K_{i-1})$, giving the second statement of \eqref{list:K}. We apply \cref{indstep} with $G_i, N_i, K_i, \delta = \frac{1}{i}$ and the finite set $A$ being the ball of radius $i$ in $N_i$. This gives a normal subgroup $M_i < G_i$, we let $G_{i+1} \coloneqq G_i/M_i$ and let $\pi_i$ be the quotient map. Then \cref{indstep} gives directly points \eqref{list:hvs} and \eqref{list:scl}, and the first statement of \eqref{list:K}.

We claim that $N_\infty$ and $G_\infty$ have the properties in the statement of \cref{norms}. Being quotients of $G_1$ and $N_1$ respectively, $G_\infty$ and $N_\infty$ are finitely generated. $G_\infty$ is torsion-free because every $G_i$ is. To see that $G_\infty$ is residually finite, let $1 \neq g \in G_\infty$. Let $i$ be the length of $g$, so we may choose a lift to $G_i$, which we also denote by $g$, which again has length $i$. By point \eqref{list:rf}, $g$ maps to a non-trivial element of $G_i/K_i$. Moreover, point \eqref{list:K} and induction imply that $G_i \to G_i/K_i$ factors through every $G_j : j \geq i$, hence through $G = G_\infty$. This gives a homomorphism of $G$ to a finite group $G_i / K_i$ such that $g$ maps to a non-trivial element. The first and third item of the theorem follow from the choice of $Q$, by \cref{bridson} and \cref{higman}. The second item follows from point \eqref{list:scl}, together with the monotonicity of $\nu$ under quotients.
\end{proof}

Taking $\nu$ to be commutator length, and using Bavard duality \cite{bavard}, we obtain half of our main theorem.

\begin{corollary}
\label{scl}

There exists a torsion-free finitely generated residually finite group $G_\infty$, with a finitely generated normal subgroup $N_\infty$, such that:
\begin{itemize}
\item $N_\infty \to G_\infty$ induces an isomorphism of profinite completions;
\item $\scl_{N_\infty}(g) = 0$ for all $g \in N_\infty$, hence every quasimorphism on $N_\infty$ is bounded;
\item $G_\infty$ has an infinite-dimensional space of homogeneous quasimorphisms, hence $\scl_{G_\infty}$ is unbounded.
\end{itemize}
\end{corollary}

\begin{remark}
\label{stablewlength}

Let $\nu$ be $w$-length, for $w$ a non-trivial reduced word in the commutator subgroup of a free group. Then similarly \cref{norms} implies that stable $w$-length is not a profinite invariant: indeed we only need the easy direction of Bavard duality \cite{bavard}, to deduce that a group with a non-zero homogeneous quasimorphism has unbounded stable $w$-length. If instead $w$ does not belong to the commutator subgroup, then stable $w$-length vanishes on every group \cite[Lemma 2.13]{stablewlength}.
\end{remark}

\section{Free subgroups and fixed point properties}
\label{s:NL}

In this section we refine the construction of \cref{norms} to strengthen the rigidity imposed on the subgroup $N_\infty$. Our main goal is a purely algebraic statement.

\begin{theorem}
\label{algebraic:NL}

There exists a torsion-free finitely generated residually finite group $G_\infty$, with a finitely generated normal subgroup $N_\infty$, such that:
\begin{itemize}
\item $N_\infty \to G_\infty$ induces an isomorphism of profinite completions;
\item $N_\infty$ is perfect, has vanishing $\scl$, no non-abelian free subgroups, and it does not virtually map onto $\Z$;
\item $G_\infty$ surjects onto Higman's group.
\end{itemize}
\end{theorem}

\subsection{Fixed point properties}
\label{s:fixedpoints}

Following \cite{NL}, we say that a group has \emph{property NL} (short for \emph{No Loxodromics}) if it cannot act on a hyperbolic space with a loxodromic element. This is the strongest form of rigidity for actions on hyperbolic spaces, indeed every countably infinite group can act elliptically on a point, or parabolically on an unbounded hyperbolic space (e.g. a combinatorial horoball on a Cayley graph in the finitely generated case \cite{groves:manning}, or a tree in the infinitely generated case \cite[Th{\'e}or{\`e}me 15]{trees}). Examples of groups with property NL include high rank lattices \cite{NL:lattices1, NL:lattices2}, Thompson-like groups \cite{NL:V, NL}, and some torsion-free Tarski monsters \cite{tarski}.

We also consider \emph{property FW$_\infty$}: the fixed point property for actions on finite-dimensional CAT(0) cube complexes, of any dimension. This was introduced by Barnhill and Chatterji in \cite{guido} with the name ``property FW'', however since then the name has been used to encompass also actions on infinite-dimensional CAT(0) cube complexes \cite{cornulier}, so we use the convention of \cite{NL:V}. Also this property is known for high rank lattices \cite{CFI, elia}, and many groups with property NL \cite[Section 6]{NL}.

Before proving \cref{algebraic:NL}, let us see how it is sufficient to deduce these properties and thus the main theorem.

\begin{proof}[Proof of the Main Theorem]
We claim that the groups $G_\infty$ and $N_\infty$ from \cref{algebraic:NL} have the properties of the main theorem. $G_\infty$ has an infinite-dimensional space of homogeneous quasimorphisms, and a general type action on a tree, both factoring through $Q$ (\cref{higman}). It follows from Bavard duality \cite{bavard} that it has unbounded $\scl$.

$N_\infty$ has vanishing $\scl$, hence no unbounded quasimorphisms by Bavard duality \cite{bavard}. We are also assuming that it is perfect and has no non-abelian free subgroups. It remains to show that it has property NL and property FW$_\infty$.

The argument for property NL is well-known, and we refer the reader to \cite{NL} for more details on each point. The absence of non-abelian free subgroups prevents general type actions, via ping-pong. The absence of unbounded quasimorphisms prevents focal and oriented lineal actions, via the Busemann pseudocharacter \cite[Section 4.1]{manning}. The absence of index-$2$ subgroups prevents non-oriented lineal actions, via the permutation of the two ends. Hence by the classification of actions on hyperbolic spaces \cite[Section 8.2]{gromov}, every action is either elliptic or horocyclic, so it has no loxodromics.

For property FW$_\infty$ we use a criterion of Genevois \cite[Theorem 5.1]{NL:V}: A group has property FW$_\infty$, provided that no finite index subgroup has a general type action on a hyperbolic space, or maps onto $\Z$. The first condition is ensured by the absence of non-abelian free subgroups as we saw above, and the second condition is assumed.
\end{proof}

\subsection{The inductive step}

The following is an analogue of \cref{indstep} that eliminates free subgroups.

\begin{lemma}
\label{indstep:NL}
	Let $G$ be a torsion-free, non-elementary hyperbolic, virtually special group. Let $N < G$ be a non-trivial normal subgroup, let $K < G$ be a finite-index normal subgroup, and let $A \subset N$ be a finite subset.
	
	Then there exists a normal subgroup $M < G$ such that
	\begin{itemize}
	\item $M \subset N \cap K$;
	\item $G/M$ is torsion-free, hyperbolic and virtually special;
	\item No pair of elements $a, b \in A$ freely generates a free subgroup of $G/M$.
	\end{itemize}
\end{lemma}

\begin{proof}
By induction, it suffices to prove this in the case when $A = \{a, b\}$ consists of two distinct non-trivial elements. Let $H \coloneqq \langle a, b \rangle$; if $H \ncong F_2$ we can take $M = \{ 1 \}$. Otherwise, \cref{random} gives elements $x, y \in \langle a, b \rangle$ such that $P \coloneqq \langle x, y \rangle < H$ is a free group of rank $2$ that is quasiconvex and malnormal in $G$. By \cref{bowditch}, $G$ is hyperbolic relative to $P$. Let $P' < P$ be the finite-index normal subgroup given by \cref{MSQT}.

Let $m \in \N$ be sufficiently large so that $x^m, y^m \in K \cap P'$. Choose $w \in \langle x^m, y^m \rangle$ to be a $C'(\frac{1}{6})$-small cancellation word in $P$ that is not a proper power. Setting $W \coloneqq \normal{w}_P$ and $M \coloneqq \normal{w}_G = \normal{W}_G$, we see that $M \subset N \cap K$, and that $G/M$ is torsion-free, hyperbolic and virtually special, just as in the proof of \cref{indstep}. Moreover, the map $\langle a, b \rangle_G \to G/M$ is not injective, since it has $w$ in its kernel; because free groups are Hopfian, this shows that $\langle a, b \rangle_{G/M} \ncong F_2$.
\end{proof}

The following is an analogue of \cref{indstep} that eliminates virtual maps onto $\Z$.

\begin{lemma}
\label{indstep:FW}
	Let $G$ be a torsion-free, non-elementary hyperbolic, virtually special group. Let $N < G$ be a non-trivial normal subgroup, let $K < G$ be a finite-index normal subgroup, and let $L < N$ be a non-elementary finitely generated subgroup.
	
	Then there exists a normal subgroup $M < G$ such that
	\begin{itemize}
	\item $M \subset N \cap K$;
	\item $G/M$ is torsion-free, hyperbolic and virtually special;
	\item The image of $L$ in $G/M$ has finite abelianisation.
	\end{itemize}
\end{lemma}

\begin{proof}
Because $L$ is finitely generated, by induction it suffices to show that, for an element $a \in L$, we can find a normal subgroup $M < G$ satisfying the first two items in the statement, such that $a$ has finite order in the abelianisation of $\bar{L}$, the image of $L$ in $G/M$. To apply induction, we need to moreover ensure that $\bar{L}$ is a non-elementary subgroup of $G/M$.

Write $a = b^r$, where $b$ is primitive. \cref{random} gives elements $x, y \in L$ such that $P \coloneqq \langle b, x, y \rangle < G$ is a free group of rank $3$ that is quasiconvex and malnormal in $G$. By \cref{bowditch}, $G$ is hyperbolic relative to $P$. Let $P' < P$ be the finite-index normal subgroup given by \cref{MSQT}.

Let $m \in \N$ be sufficiently large so that $a^m \in K \cap P'$. We proceed as in the proof of \cref{indstep}, adding a small cancellation relation of the form
\[w \coloneqq a^{-mq} \prod\limits_{i = 1}^p \gamma_i a^m \gamma_i^{-1},\]
where we moreover choose $p \neq q$. This way, $a$ has order dividing $m|q-p| > 0$ in the abelianisation of $\bar{L}$. Finally, $\bar{L}$ is still non-elementary, since it contains the non-elementary subgroup $\langle a, x, y \rangle \normal{w}_P / \normal{w}_P$. The rest follows as in \cref{indstep}.
\end{proof}

\subsection{Proof of \cref{algebraic:NL}}

We start as in the proof of \cref{norms}, choosing $\nu$ to be the commutator length $\cl$. Now the $\Gamma$ given by \cref{w:finite} is a perfect group, and $Q$ is still Higman's group (\cref{higman}). \cref{macarena} gives a short exact sequence
\[1 \to N_1 \to G_1 \to Q \to 1,\]
where $G_1$ is hyperbolic, torsion-free, virtually special, and $N_1$ is perfect.

We construct by induction short exact sequences
\[\begin{tikzcd}
	& \vdots & \vdots & \vdots \\
	1 & {N_i} & {G_i} & Q & 1 \\
	1 & {N_{i+1}} & {G_{i+1}} & Q & 1 \\
	& \vdots & \vdots & \vdots \\
	1 & {N_\infty} & {G_\infty} & Q & 1
	\arrow[two heads, from=1-2, to=2-2]
	\arrow[two heads, from=1-3, to=2-3]
	\arrow[no head, from=1-4, to=2-4]
	\arrow[shift left, no head, from=1-4, to=2-4]
	\arrow[from=2-1, to=2-2]
	\arrow[from=2-2, to=2-3]
	\arrow[two heads, from=2-2, to=3-2]
	\arrow[from=2-3, to=2-4]
	\arrow["{\pi_i}", two heads, from=2-3, to=3-3]
	\arrow[from=2-4, to=2-5]
	\arrow[shift left, no head, from=2-4, to=3-4]
	\arrow[no head, from=2-4, to=3-4]
	\arrow[from=3-1, to=3-2]
	\arrow[from=3-2, to=3-3]
	\arrow[two heads, from=3-2, to=4-2]
	\arrow[from=3-3, to=3-4]
	\arrow[two heads, from=3-3, to=4-3]
	\arrow[from=3-4, to=3-5]
	\arrow[shift left, no head, from=3-4, to=4-4]
	\arrow[no head, from=3-4, to=4-4]
	\arrow[two heads, from=4-2, to=5-2]
	\arrow[two heads, from=4-3, to=5-3]
	\arrow[shift left, no head, from=4-4, to=5-4]
	\arrow[no head, from=4-4, to=5-4]
	\arrow[from=5-1, to=5-2]
	\arrow[from=5-2, to=5-3]
	\arrow[from=5-3, to=5-4]
	\arrow[from=5-4, to=5-5]
\end{tikzcd}\]
and finite-index normal subgroups $K_i < G_i$ with the following properties.
\begin{enumerate}
\item\label{newlist:hvs} Each $G_i$ is torsion-free, hyperbolic and virtually special (hence residually finite);
\item\label{newlist:scl} For every element $a \in N_i$ of length at most $i$, it holds $\scl_{N_i}(a) < \frac{1}{i}$;
\item\label{newlist:free} No pair of elements $a, b \in N_i$ of length at most $i$ freely generates a free group;
\item\label{newlist:Z} If $L_i$ denotes the intersection of all index-$i$ subgroups of $N_i$, then $L_i$ has finite abelianisation;
\item\label{newlist:K} $\ker(\pi_{i-1}) < N_{i-1} \cap K_{i-1}$ and $K_i < \pi_{i-1}(K_{i-1})$;
\item\label{newlist:rf} The ball of radius $i$ in $G_i$ maps injectively into $G_i/K_i$.
\end{enumerate}
This is achieved as in the proof of \cref{norms}, but in three steps: first we apply \cref{indstep} to ensure \eqref{newlist:scl}; then we apply \cref{indstep:NL} to ensure \eqref{newlist:free}; and finally we apply \cref{indstep:FW}\footnote{We can apply \cref{indstep:FW}, because the intersection of all subgroups of a given index in a finitely generated group has finite index, hence is finitely generated.} to ensure \eqref{newlist:Z}. Note that each of these three properties passes to quotients, so in the last step all three points are satisfied. The remaining points follow as in the proof of \cref{norms}.

We claim that $N_\infty$ and $G_\infty$ have the properties in the statement of \cref{algebraic:NL}. Clearly $G_\infty$ surjects onto $Q$, and $N_\infty$ is perfect. The vanishing of $\scl_{N_\infty}$, finite generation, torsion-freeness, residual finiteness, and the profinite isomorphism, all follow exactly as in the proof of \cref{norms}.

To see that $N_\infty$ has no non-abelian free subgroups, it suffices to show that no pair of elements $a, b \in N_\infty$ freely generates a free group. Suppose that $a$ and $b$ have length at most $i$, then they admit lifts to $N_i$, which we also denote by $a, b$, which again have length at most $i$. By point \eqref{newlist:free}, $\langle a, b \rangle_{N_i} \ncong F_2$; because free groups are Hopfian, $\langle a, b \rangle_{N_\infty} \ncong F_2$.

Finally, suppose by contradiction that there exists a finite-index subgroup $L < N_\infty$ that maps onto $\Z$. Up to passing to a deeper finite-index subgroup, we may assume that there exists $i$ such that $L$ is the intersection of all index-$i$ subgroups of $N_\infty$. Therefore $L_i$ is a subgroup of the preimage of $L$ in $N_i$. Since $L_i$ has finite index, the surjection $L \to \Z$ pulls back to a non-zero map $L_i \to \Z$, contradicting \eqref{newlist:Z}. \qed

\section{Concluding remarks}
\label{s:remarks}

On the flexibility of our construction:

\begin{remark}
\label{flexibility}

Like many others involving iterated small cancellation or iterated group-theoretic Dehn filling, our construction is quite flexible. Here are some other properties that can be imposed.

\begin{itemize}
\item Using the last part of \cite[Theorem 1.1]{osin} in the inductive step, we can ensure that in the construction of $G_\infty$ the quotients $G_i \to G_{i+1}$ are injective on a given finite set. Taking this into account in the proof, ensuring that at each step the injectivity radius is large compared to the hyperbolicity constant, makes the limit group $G_\infty$ lacunary hyperbolic \cite[Theorem 1.1]{lacunary}.

\item Adapting \cref{w:finite} appropriately, we can pick as a starting point a hyperbolic virtually special group $\Gamma$ of larger cohomological dimension. Using the full power of \cite[Theorem 1.1]{macarena}, we can then ensure that $G_1$ has arbitrarily large cohomological dimension. \cite[Corollary 2.2]{PS} implies that the cohomological dimension of each $G_i$ remains high and bounded, and hence so does the cohomological dimension of $G_\infty$, as in \cite{ffsun}.

\item As we saw in the proof of torsion-freeness in \cref{indstep}, we can ensure that the quotients $G_i \to G_{i+1}$ satisfy an arbitrarily good geometric small cancellation condition. This allows to apply \cite[Theorem H]{tarski} at each step, just as in the construction of torsion-free Tarski monsters with controlled first-order theory \cite[Theorem 7.10]{tarski}. This way, $G_\infty$ may be chosen so that its two-quantifier first-order theory satisfies $\mathrm{Th}_{\forall \exists}(G_\infty) = \mathrm{Th}_{\forall \exists}(G_\infty * \Z)$, and hence its positive first-order theory $\mathrm{Th}^+(G_\infty)$ coincides with that of a non-abelian free group, by \cite[Theorem F]{montse}\footnote{See \cite[Corollary 7.15]{tarski} for a more detailed argument}.
\end{itemize}
\end{remark}

On a limitation of the method:

\begin{remark}
To ensure residual finiteness of the direct limit $G_\infty$, we used in a crucial way that the Dehn filling in the inductive step can be arranged inside a given finite-index subgroup. This is also necessary in the proof of e.g. \cref{indstep} itself, in order to apply \cref{MSQT}. For our purposes this was not an issue, as it was sufficient to work with a suitable power.

This however becomes an issue if one tries to arrange $N_\infty$ to be \emph{uniformly perfect}, that is its commutator length is bounded. This is much stronger than the vanishing of $\scl$, and is a condition that needs to be verified element-wise: passing to a suitable power is not sufficient. In fact, the last item of \cref{flexibility} seems to suggest that $N_\infty$ would have trivial positive theory in a construction of this type, which would imply in particular that it is not uniformly perfect.

As such, we do not know whether uniform perfectness is a profinite invariant\footnote{Note that, in all finitely generated profinite groups, the derived subgroup has bounded width \cite[Theorem 1.6]{nikolov:segal}, in particular all finitely generated perfect profinite groups are uniformly perfect.}. A related problem is whether bounded generation is a profinite invariant. Recent breakthroughs in the world of linear groups \cite{bddgen, conjwidth} suggest that it might be more promising to look among high rank lattices.
\end{remark}

On actions on CAT(0) cube complexes:

\begin{remark}
\label{FW}

We proved that property FW$_\infty$ is not a profinite invariant. We do not know if \emph{property FW}, the fixed point property for actions on (possibly infinite-dimensional) CAT(0) cube complexes \cite{cornulier}, is a profinite invariant. However we think this construction does not help. The groups $G_i$ act properly and cocompactly on CAT(0) cube complexes, so in fact it should even be possible to ensure that $G_\infty$, and thus $N_\infty$, acts properly on a CAT(0) cube complex.

In another direction, we do not know whether proper actions on CAT(0) cube complexes are a profinite invariant. Note that cocompact cubulation is \emph{not} a profinite invariant, for instance in the classical Rips construction \cite{rips}, the overgroup is cocompactly cubulated (\cref{sc}), while the normal subgroup, being infinitely presented, is not.
\end{remark}

On the results of \cite{profinite:hb}:

\begin{remark}
\label{h2b}
If $G$ is a group, and $Q(G)$ denotes the space of homogeneous quasimorphisms, then there is a short exact sequence:
\[0 \to H^1(G) \to Q(G) \to H^2_b(G) \to H^2(G),\]
where $H^2_b(G)$ denotes the second bounded cohomology of $G$ \cite[Proposition 2.8]{frigerio}.
Echtler--Kammeyer \cite{profinite:hb} proved that there exist pairs of high rank lattices $\Lambda_0, \Lambda_1$ with isomorphic profinite completions such that $H^2_b(\Lambda_i)$ is $i$-dimensional, thus proving that the vanishing of second bounded cohomology is not a profinite invariant. By work of Burger--Monod and Iozzi \cite{burger:monod1, burger:monod2, iozzi:appendix}, such lattices have no unbounded quasimorphisms, hence their construction does not recover our main theorem.

Conversely, our construction does not recover their result. On the one hand, since $G_\infty$ is finitely generated and has an infinite-dimensional space of homogeneous quasimorphisms, we see that it has infinite-dimensional second bounded cohomology. On the other hand, the vanishing of $\scl_{N_\infty}$ only tells us that the second bounded cohomology of $N_\infty$ embeds into its second cohomology. However, the second cohomology of groups constructed via iterated group-theoretic Dehn filling tends to be infinite-dimensional \cite[Remark 3.5]{ffsun}, and we do not know how to identify which classes are bounded.

It remains open whether having infinite-dimensional second bounded cohomology is a profinite invariant.
\end{remark}

On free subgroups:

\begin{remark}
\label{freesub}

In \cite{profinite:amenable} it is shown that there exists a Grothendieck pair where the subgroup is amenable but the overgroup contains a non-abelian free subgroup. This shows that both amenability and the existence of a non-abelian free subgroup are not profinite invariants. Our main theorem recovers the latter fact: indeed $G_\infty$ has a non-abelian free subgroup, via ping-pong. However it does not recover the result on amenability. In fact, the construction can be arranged so that $N_\infty$ is not amenable\footnote{The existence of non-amenable groups without free subgroup was the subject of von Neumann's conjecture (stated by Day \cite{day}) at the origin of the subject; by now we have several constructions \cite{vN, monod}.}.

Indeed, recall that the construction starts with a short exact sequence
\[1 \to N_1 \to G_1 \to Q \to 1,\]
where $G_1$ is torsion-free hyperbolic and $N_1$ is non-trivial and normal. By small cancellation theory on hyperbolic groups \cite{residualing}\footnote{See \cite[Theorem 2.4]{osin:RH} for a precise statement, although in the more general context of relatively hyperbolic groups.}, since $N_1$ is a non-elementary subgroup of $G_1$, there exists a torsion-free hyperbolic quotient $G_1 \twoheadrightarrow H$ such that the restriction $N_1 \to H$ is surjective. Taking a free product of $H$ with a torsion-free hyperbolic group with property (T), and applying small cancellation again, we obtain a quotient $H \twoheadrightarrow L$ that is torsion-free, hyperbolic, and has property (T). There exists an integer $n \in \N$ such that the $n$-Burnside quotient $B$ of $L$ is infinite \cite{burnside1, burnside2}.

Each quotient $G_i \twoheadrightarrow G_{i+1}$ in \cref{algebraic:NL} involves an application of Lemmas \ref{indstep}, \ref{indstep:NL} and \ref{indstep:FW}, all of which only add relations that are products of $m_i$-th powers of elements in $N_i$. Here $m_i$ is chosen so that every $m_i$-th power in $N_i$ belongs to a specific finite-index subgroup of $N_i$, so we may ensure that $m_i$ is a multiple of $n$. It follows that $N_\infty$ surjects onto the $n$-Burnside quotient of $N_1$, hence also onto $B$, since $L$ is a quotient of $N_1$ and $B$ is the $n$-Burnside quotient of $L$. Finally, $B$ is an infinite group with property (T), hence it is non-amenable. We conclude that $N_\infty$, by virtue of having a non-amenable quotient, is not amenable.
\end{remark}

\footnotesize

\bibliographystyle{amsalpha}
\bibliography{ref}

\vspace{0.5cm}

\normalsize

\noindent{\textsc{Department of Mathematics, Heriot-Watt University, Edinburgh, UK}}

\noindent{\textit{E-mail address:} \texttt{f.fournier-facio@hw.ac.uk}}

\end{document}